\documentclass[11pt]{amsart}
\usepackage{amssymb, amscd, amsmath, amsthm, epsf, epsfig, latexsym,psfrag}

\newcommand{\CP}{\mathbb{CP}}
\newcommand{\bCP}{\overline{\mathbb{CP}}}

\newcommand{\Kahler}{K\"{a}hler~}

\newcommand{\TryPackage}[3]{\IfFileExists{#1.sty}{\usepackage{#1}#2}{#3}
}
\TryPackage{mathrsfs}{\renewcommand{\mathcal}{\mathscr}}{%
        \TryPackage{eucal}{}{}}

\newcommand{\ZZ}{{\mathbb Z}}
\newcommand{\RR}{{\mathbb R}}

\usepackage{latexsym}
\usepackage{amsmath}
\usepackage{amsthm}
\usepackage{amscd}
\usepackage{xypic} 
\usepackage{amsfonts}
\usepackage{amssymb}
\usepackage{graphicx}
\usepackage{psfrag}
\usepackage{comment}

\newtheorem{df}{Definition}
\newtheorem{thm}[df]{Theorem}

\newtheorem{lem}[df]{Lemma}
\newtheorem{prop}[df]{Proposition}

\begin{document}

\title[Coisotropic Luttinger surgery]{Coisotropic Luttinger surgery and some new symplectic 6-manifolds with vanishing canonical class}

\author{Scott Baldridge}
\author{Paul Kirk}
\date{May 9, 2011}

\thanks{The first   author  gratefully acknowledges support from the
NSF  grant DMS-0748636. The second author gratefully acknowledges support from the
NSF  grant DMS-1007196}

\address{Department of Mathematics, Louisiana State University \newline
\hspace*{.375in} Baton Rouge, LA 70817} 
\email{\rm{sbaldrid@math.lsu.edu}}

\address{Department of Mathematics, Indiana University \newline
\hspace*{.375in} Bloomington, IN 47405} 
\email{\rm{pkirk@indiana.edu}}


\maketitle

\begin{abstract} We  introduce a  surgery operation on symplectic manifolds called coisotropic Luttinger surgery, which generalizes 
 Luttinger surgery on Lagrangian tori in symplectic 4-manifolds  \cite{Lut}.  We use it to produce infinitely many distinct  symplectic non-\Kahler 
 6-manifolds $X$ with $c_1(X)=0$   which are not of the form $M\times F$ for $M$ a symplectic 4-manifold and $F$ a closed surface. 
\end{abstract} 


\section{Introduction}

\medskip

In this article we  introduce a  surgery operation on symplectic manifolds called co\-isotropic Luttinger surgery, which generalizes 
 Luttinger surgery on Lagrangian tori in symplectic 4-manifolds  \cite{Lut, ADK}.  We use it to produce infinitely many distinct  symplectic non-\Kahler 
 6-manifolds $X$ with $c_1(X)=0$   which are not of the form $M\times F$ for $M$ a symplectic 4-manifold and $F$ a closed surface.
 
 \bigskip

\begin{thm}\label{main}  Coisotropic surgery on 4-tori in $T^6$ produces  an infinite family of pairwise non-homotopy equivalent  
 closed symplectic $6$-manifolds  $X_{n}$ with $c_1(X_{n})=0$, Euler characteristic $\chi(X_n)=0$, and  Betti numbers satisfying $b_1(X_n)=3$, $b_2(X_n)\leq 18, $ and $b_3(X_n)\leq 32$.  None of the  manifolds $X_{n}$ are symplectomorphic to $M\times F$ for a symplectic 4-manifold $M$ and surface $F$.\end{thm}

\bigskip

Coisotropic Luttinger surgery has a very simple topological description which generalizes Dehn surgery in dimension 3.  It is localized near a certain codimension two coisotropic submanifold.   This makes it useful as method to produce related families of symplectic manifolds. In Theorem \ref{main}  the 4-tori on which the surgeries are performed are not symplectic, but rather products of Lagrangian with symplectic tori.

Symplectic   manifolds $M$ with vanishing  first Chern class are known as {\em symplectic Calabi-Yau manifolds} \cite{STY, FP1}.   The famous Kodaira-Thurston 4-manifold \cite{Th} provided the first non-K\"ahler example of such a manifold, and one can produce higher dimensional examples by taking its product with a torus. 

Symplectic Calabi-Yau manifolds which do not admit Kahler structures have received attention recently (cf.~\cite{FP1, FP2,  TW2, TW1, W}).  
In dimension 6 these manifolds were introduced  by Smith, Thomas, and Yau in their paper on symplectic conifold transitions  \cite{STY} motivated by their use in producing pairs exhibiting mirror symmetry.   Their construction  involves  symplectically resolving singular complex projective 3-folds.  Other known examples   include certain   nilmanifolds (e.g.~\cite{CLG}) (of which the Kodaira-Thurston manifold is an example), and the compelling constructions of Fine-Panov \cite{FP1,FP2} which are obtained from $S^2$ bundles over 4-dimensional hyperbolic orbifolds and $S^3$ bundles over hyperbolic 3-manifolds.

\medskip
 
The examples we produce   reconstruct a few of these previously known examples,  although in this article  we take our seed manifold to be $T^6$ and hence we do not produce simply connected examples. But the essential property that we exploit is   that $T^6$ fibers in different ways.

 \medskip
In addition to its use in constructing the examples of Theorem \ref{main}, coisotropic Luttinger surgery applies in a wide range of contexts in symplectic topology. In particular, it extends to higher dimensions a codimension two symplectic surgery operation which, in concert with the symplectic sum operation \cite{Gompf},  has already had significant impact in 4-dimensional smooth topology c.f.~\cite{BK2}.     Moreover, coisotropic  surgery is localized near a 
 submanifold and so one can understand the change in homotopy  invariants by standard Mayer-Vietoris arguments.
We expect the process to have interesting applications outside the context of symplectic Calabi-Yau manifolds.  We touch on some potential further applications in the last section.

\section{Coisotropic Luttinger surgery}

 We describe the construction, which consists of removing $D^2\times T^2\times \Sigma$ from a symplectic $2n$-manifold and regluing by an appropriate diffeomorphism of the boundary.

  Let $ D^2_\epsilon $ denote the closed disk in $\RR^2$ of radius $\epsilon$ and coordinates $x, y$, hence 1-forms $dx, dy$.  Let $T^2=S^1\times S^1=\RR^2/\ZZ^2$ denote the 2-torus  with coordinates $e^{iz}, e^{iw}$ and its global 1-forms $dz,dw$ (descended from $\RR^2$).

Suppose we are given  a $2n-4$ dimensional manifold $\Sigma$  
and    a family $ \omega_{\Sigma;(d,t)} $   of symplectic forms on $\Sigma$ parametrized by $(d,t)\in D^2_\epsilon\times T^2$.  Then 
 $D^2_\epsilon \times T^2\times \Sigma$ inherits the symplectic form 
 \begin{equation}
\label{omega}\omega=dx~dz + dw~dy +  \omega_{\Sigma;(d,t)}.
\end{equation}
The parallel submanifolds  $\{(x,y)\}\times T^2\times\Sigma,~(x,y)\in D^2_\epsilon$ are cosiotropic with respect to $\omega$.

We extend  Luttinger surgery  \cite{Lut} as follows.  
Suppose that $(X,\omega_X)$ is a symplectic $2n$-manifold and $\Sigma$ is a closed $2n-4$-dimensional smooth manifold. 
Suppose one is given an embedding 
$$e:D^2_\epsilon \times T^2\times \Sigma \hookrightarrow X$$
so that the pulled back symplectic form $e^*(\omega_X)$ satisfies
$$e^*(\omega_X) = \omega 
$$
for $\omega$ the form defined in Equation (\ref{omega}).

Fix an integer $k$   and let $$\phi_k:  (D^2_{\epsilon}\setminus D^2_{2\epsilon/3})\times T^2\times \Sigma\to 
 (D^2_{\epsilon}\setminus D^2_{2\epsilon/3})\times T^2\times \Sigma$$
 denote the diffeomorphism given in polar coordinates on $D^2_{\epsilon}\setminus D^2_{5\epsilon/6}$  by
 $$\phi_k(re^{i\theta}, e^{iz}, e^{iw},\sigma)=(re^{i\theta}, e^{iz}, e^{i(w+ k\theta)},\sigma).$$

\begin{lem} \label{extends} The symplectic form $\phi_k^*(\omega)$ extends to a symplectic form on $  D^2_{\epsilon} \times T^2\times \Sigma$
\end{lem}

\begin{proof}

One computes 
$$ (\phi_k)_*(\tfrac{\partial}{\partial r})=\tfrac{\partial}{\partial r},
\  (\phi_k)_*(\tfrac{\partial}{\partial \theta})=\tfrac{\partial}{\partial \theta}+k  \tfrac{\partial}{\partial w}, 
\  (\phi_k)_*(\tfrac{\partial}{\partial z})=\tfrac{\partial}{\partial z},
 \  (\phi_k)_*(\tfrac{\partial}{\partial w})=\tfrac{\partial}{\partial w}.$$
Hence
\begin{equation*}\label{forms}
\phi_k^*(dr)=dr, ~\phi_k^*(d\theta)=d\theta, ~\phi_k^*(dz)=dz,~ \text{and} ~\phi_k^*(dw)=dw+   kd\theta.
\end{equation*}
Switching back to  Cartesian coordinates on $D^2_\epsilon$ yields 
\begin{equation}\label{forms1}
\phi_k^*(dx)=dx, ~\phi_k^*(dy)=dy, ~\phi_k^*(dz)=dz,~ \text{and} ~\phi_k^*(dw)=dw+  \tfrac{k}{  x^2+y^2 }(xdy-ydx).
\end{equation}
Hence 
$$\phi_k^*(\omega)=\omega - \tfrac{k}{ x^2+y^2 }~y~dx~dy.$$

Fix a radially symmetric smooth function $f:D^2_\epsilon\to \RR$  which equals  $0$ for $\sqrt{x^2+y^2}\leq   \tfrac{ \epsilon }{3}$ and $  \tfrac{1}{ x^2+y^2 }$   for $\sqrt{x^2+y^2}\ge\tfrac{2\epsilon}{3}$.     Then the 2-form 
$$\alpha=-ky f(x,y)~dxdy$$  on $D^2_\epsilon\times T^2\times \Sigma$
  is closed since it is pulled back from a 2-form on $D^2_\epsilon$. 
 Thus 
 \begin{equation}\label{symplectic}
 \tilde\omega= \omega + \alpha= \omega- kyf~dxdy
 \end{equation} is closed.  It agrees with $\phi_k^*(\omega)$ on $(D^2_{\epsilon}\setminus D^2_{2\epsilon/3})\times T^2\times \Sigma$ and agrees with $\omega$ on $ D^2_{\epsilon/3} \times T^2\times \Sigma$. On $(D^2_{2\epsilon/3}\setminus D^2_{\epsilon/3})\times T^2\times \Sigma,$ $\alpha^2=0$  and $\omega \wedge \alpha=\omega_\Sigma\wedge \tau$ 
and so
$\tilde\omega^n=\omega^n + \omega_\Sigma^{n-1}\alpha=\omega^n$   
and hence $\tilde \omega$ is non-degenerate, i.e.~ a symplectic form.
\end{proof}

Construct a new manifold $X'$ as the union with identifications
$$X'=  \big( (X\setminus e(D_{2\epsilon/3}\times T^2\times \Sigma))\sqcup \big( D^2_\epsilon\times T^2\times \Sigma)\big)/\sim$$
where the points 
  $e(re^{i\theta}, e^{iz},e^{iw},\sigma)\in X$ and $ \phi_k(re^{i\theta}, e^{iz},e^{iw},\sigma)\in D^2_\epsilon\times T^2\times \Sigma $ are identified provided $  \tfrac{2\epsilon}{3}\leq r\leq \epsilon.$   Lemma \ref{extends} shows that the symplectic form on $X\setminus e(D_{2\epsilon/3}\times T^2\times \Sigma)$ extends to a symplectic form on $X'$.

Since this construction depends on an coisotropic submanifold instead of a Lagrangian submanifold, we  say $X'$ is obtained from $X$ by $\frac{1}{k}$  {\em coisotropic Luttinger surgery  along $T^2\times\Sigma\subset X$}.  If  $k=0$ then clearly $X'=X$.

\bigskip

As a smooth manifold, $X'$ can be described as the manifold obtained by removing $D^2\times T^2\times \Sigma$ from $X$ and regluing using the restriction of $\phi_k$ to the boundary:
\begin{equation}\label{psi}
\psi_k:S^1\times T^2\times \Sigma\to S^1\times T^2\times \Sigma, \ 
\psi_k(e^{i  \theta}, e^{iz}, e^{iw},\sigma)=(e^{i  \theta}, e^{iz}, e^{i(w+k \theta)},\sigma).\end{equation}

The   following proposition is  well known in the case when $n=2$, that is, for Luttinger surgery on 4-manifolds.  

\begin{prop} If $X'$ is obtained from $X$ by $\frac{1}{k}$ coisotropic Luttinger surgery along $T^2\times\Sigma\subset  X$, then the
Euler characteristic is unchanged, 
$\chi(X')=\chi(X).$   When $\dim(X)=4\ell$,   the signature  is unchanged, $\sigma(X')=\sigma(X)$.  The fundamental group of $X'$ is the quotient of $\pi_1(X\setminus (T^2\times \Sigma))$ by the normal subgroup generated by the circle $\psi_k(\partial D^2_\epsilon \times \{p\})$

\end{prop}
\begin{proof}
Using the Mayer-Vietoris sequence one sees that the Euler characteristic is unchanged, 
$\chi(X')=\chi(X).$
When $n$ is even Novikov additivity shows that the signature  is unchanged, $\sigma(X')=\sigma(X)$.  

Give $ T^2\times\Sigma$ a handlebody structure with handles of index $0,1,\cdots, 2n-3$ and a single $2n-2$-handle. The product   $D^2 \times T^2\times\Sigma$  has a handlebody structure obtained by taking the product of $D^2 $ with each handle, and in particular has a single $2n-2$-handle.    Turning the handle decomposition upside down shows that $D^2 \times T^2\times\Sigma$ is obtained from $X \setminus(D^2_\epsilon\times T^2\times \Sigma)$ by attaching a single 2-handle along the attaching circle $\psi_k(\partial D^2_\epsilon  \times (1,1,\sigma_0))$, and then adding handles of  index greater than 2.

The Seifert-Van Kampen theorem implies that 
$$\pi_1(X')=\pi_1( X \setminus(D^2_\epsilon\times T^2\times \Sigma))/N\big(\psi_k(\partial D^2_\epsilon \times \{p\})\big).$$
\end{proof}

Calculating $\pi_1(X\setminus T^2\times\Sigma)$ in terms of   $\pi_1(X)$ and the embedding $T^2\times \Sigma\subset X$  can be a challenge, since $T^2\times \Sigma$ has codimension two in $X$. In our main application below we will content  ourselves with the easier task of computing $H_1(X\setminus T^2\times\Sigma)$ and then $H_1(X')$.

\section{Producing  symplectic 6-manifolds with $c_1=0$}

Consider $X=T^6=T^2\times T^2\times T^2$, the 6-torus. Endow $X$ with the symplectic form $\omega_X=dx_1dy_1+dx_2dy_2+dx_3dy_3$. We can find four  disjoint embeddings of $T^2\times T^2$ in $X$ with the properties we need to apply the construction of the previous section:
\begin{equation}\label{embeds}
\begin{split} e_1(e^{i z}, e^{iw},  \sigma_1,\sigma_2)=  (1, e^{iz}, e^{iw},1,   \sigma_1,\sigma_2) \\
 e_2(e^{i z}, e^{iw},  \sigma_1,\sigma_2)=  (i, e^{iz}, 1,  e^{iw}, \sigma_1,\sigma_2) \\
 e_3(e^{i z}, e^{iw},  \sigma_1,\sigma_2)=  (-1, e^{iz},    \sigma_1,\sigma_2,  e^{iw}, 1) \\
 e_4(e^{i z}, e^{iw},  \sigma_1,\sigma_2)=  (-i, e^{iz},    \sigma_1,\sigma_2,  1, e^{iw}). 
\end{split}
\end{equation}
 
These are disjoint since their first coordinates are different. Note that $e_i(T^2\times (\sigma_1,\sigma_2))$ is isotropic and $e_i((r,s)\times T^2)$ is symplectic.

For $\epsilon>0$ small, extend $e_i$ to $D^2_\epsilon\times T^2\times T^2$ by 
 \begin{equation}
\label{e1}\begin{split}
e_1(x,y,e^{iz},e^{iw},e^{is_1},e^{is_2})&=(e^{ix}, e^{iz}, e^{iw}, e^{iy},e^{is_1},e^{is_2}),\\
e_2(x,y,e^{iz},e^{iw},e^{is_1},e^{is_2})&=(ie^{ix}, e^{iz}, e^{-iy},e^{iw}, e^{is_1},e^{is_2}),\\
e_3(x,y,e^{iz},e^{iw},e^{is_1},e^{is_2})&=(-e^{ix}, e^{iz}, e^{is_1},e^{is_2}, e^{iw},e^{iy}),\\
e_4(x,y,e^{iz},e^{iw},e^{is_1},e^{is_2})&=(-ie^{ix}, e^{iz}, e^{is_1},e^{is_2}, e^{-iy},e^{iw}).
\end{split}\end{equation}
Then $$e_i^*(dx_1dy_1+dx_2dy_2+dx_3dy_3)=dx~dz+dw~dy +d\sigma_1~d\sigma_2=\omega$$ for each $i$.

One can find many more embeddings of $T^2\times T^2$ by precomposing $e_i$ by a diffeomorphism $\tau:T^2\times T^2\to T^2\times T^2$ of the form 
\begin{equation} \label{tau}
\tau(e^{i z}, e^{iw},  \sigma_1,\sigma_2)=(e^{i (pz+qw)}, e^{i(rz+sw)},  \sigma_1,\sigma_2)
\end{equation}
for integers $p,q,r,s$ satisfying $ps-qr=1$. 
Identify $\tau$ with the corresponding matrix in $SL(2,\ZZ)$. Precomposing   $e_i$ by $\tau\in SL(2,\ZZ)$ and extending over $D^2_\epsilon\times T^2\times T^2$ yields another embedding with $(e_i\circ\tau)^*(dx_1dy_1+dx_2dy_2+dx_3dy_3)=\omega$, since $\det\tau=1$.

 \bigskip

Choose a surgery  parameter $k_i$ and a matrix $\tau_i\in SL(2,Z)$   for each embedding  $e_i$. Applying the coisotropic surgery procedure to $T^6$ yields a family of 6-dimensional symplectic manifolds $X_{k, \tau}$, indexed by  $(k,\tau)=(k_1,k_2,k_3,k_4;\tau_1,\tau_2,\tau_3,\tau_4)$ in the infinite set $ \ZZ^4\times(SL(2,\ZZ))^4$.  These are not all symplectically distinct; for example $SL(2,\ZZ)^3$ (and even $Sp(6,\ZZ)$) acts on this collection via its action on $T^6=T^2\times T^2\times T^2$.  But there are infinitely many distinct manifolds in this family.  The following theorem is our main result, which immediately implies Theorem \ref{main}   promised in the introduction, by taking $d_1=0,d_2=n, d_3=1$ and $d_4=1$.

\bigskip

\begin{thm}\label{result} For $(k,\tau)\in \ZZ^4\times (SL(2,\ZZ))^4$, the closed symplectic manifolds $X_{k,\tau}$ satisfy
 $c_1(X_{k,\tau})=0$.
The first homology  $H_1(X_{k,\tau})$ is the quotient of $\ZZ^6=\langle x_1,x_2,x_3,x_4,x_5,x_6\rangle$ by the subgroup generated by 
$$k_1(q_1x_2+s_1x_3), k_2(q_2x_2+s_2x_4), k_3(q_3x_2+s_3x_5),k_4(q_4x_2+s_4x_6).$$ 
Hence any abelian group of the form $$\ZZ^2\oplus\ZZ/d_1\oplus\ZZ/d_2\oplus \ZZ/d_3\oplus\ZZ/d_4$$  with $d_1,d_2,d_3,d_4$ non-negative integers can be realized as $H_1(X_{k,\tau})$ for an appropriate $(k,\tau)$.   

  If $b_1(X_{k,\tau})$ is odd then $X_{k,\tau}$ admits no K\"ahler structure.  If $b_1(X_{k,\tau})\leq 3$ then $X_{k,\tau}$ is not symplectomorphic to the product  $M\times F$ of a symplectic 4-manifold and a surface.   Finally, $b_2(X_{k,\tau})\leq 15+ b_1(X_{k,\tau})$ and 
  $b_3(X_{k,\tau})\leq 32$.
\end{thm}

The proof will take up the remainder of this section, and follows   from Theorem \ref{chern}, Proposition \ref{rels}, and Theorem \ref{product}. 

\bigskip

 We begin with the calculation of the first Chern class.

\begin{thm}\label{chern} The symplectic 6-manifolds $X_{\tau,k}$ satisfy $c_1(X_{k,\tau})=0$.
\end{thm}
\begin{proof}  Fix $k\in \ZZ$. We make use of the function $f:D^2_\epsilon\to \RR$ which equals $0$ for $\sqrt{x^2+y^2}\leq   \tfrac{ \epsilon }{3}$ and $  \tfrac{1}{ x^2+y^2 }$   for $\sqrt{x^2+y^2}\ge\tfrac{2\epsilon}{3}$. In terms of this function, define an almost complex structure 
$J_k$ acting on 1-forms on $D^2_\epsilon\times T^2\times T^2$ by
 \begin{equation*}\begin{split}
J_k(dx)&=-dz,\\ J_k(dz)&=dx,\\ J_k(dy)&= dw -kyf~dx + kxf~dy, \\
J_k(dw)&=-(1+k^2x^2f^2)~dy + k^2xyf^2~dx-k yf~dz-kxf~dw\\
J_k(d\sigma_1)&=-d\sigma_2, \\
J_k(d\sigma_2)&=d\sigma_1
\end{split}
\end{equation*}
It is routine to check that $J_k^2=-1$ and that  $J_k$ is compatible with the symplectic form
$$\omega_k=dxdz+dwdy+d\sigma_1d\sigma_2- kyf~dxdy=\omega- kyf~dxdy.
$$
Thus
\begin{eqnarray*}
s_k&=&(dx-iJ_k(dx))(dw-iJ_k(dw))(d\sigma_1-iJ_k(d\sigma_1))\\
&=&(dx+idz)(dw+idy)(d\sigma_1+id\sigma_2)+kf(xdxdy+i(ydxdz-xdydz))(d\sigma_1+id\sigma_2)
\end{eqnarray*}
 is a  section of $(3,0)$ forms. This is nowhere zero since the coefficient of $dxdw$ is 1, and hence pointwise spans the canonical bundle of $(D^2_\epsilon\times T\times T,J_k)$.

Using (\ref{forms}) one   calculates that over $(D^2_\epsilon\setminus D^2_{2\epsilon/3})\times T\times T$
\begin{equation}
\label{pullback}
\phi_k^*(J_0)=J_k, \phi_k^*(\omega)=\omega_k, ~\text{ and }~\phi^*(s_0)=s_k.
\end{equation}

Now $T^6$ is endowed with the symplectic form $\omega_X=dx_1dy_1+dx_2dy_2+dx_3dy_3$, compatible almost complex structure $J_X(dx_i)=-dy_i$, and nowhere zero section of its canonical bundle $s_X=(dx_1+idy_1)(dx_2+idy_2)(dx_3+idy_3)$.

For each $i=1,2,3,4$, $e_i^*(s_X)=s_0$, $e^*_i(J_X)=J_0$, and $e_i^*(s_X)=s_0$.  Using (\ref{pullback}) it follows that  when all the $\tau_i$ are the identity,  the almost complex structures $J_{k_i}$ and the sections $s_{k_i}$ over $D^2_\epsilon\times T\times T$
and the restrictions of $J_X$ and $s_X$ to the complement of $\sqcup_i e_i(D^2_{2\epsilon/3}\times T\times T))$ in $T^6$ patch together over $e_i((D^2_\epsilon\setminus D^2_{2\epsilon/3})\times T\times T)$ to give an almost complex structure $\tilde J$ compatible with $\tilde \omega$ and a nowhere zero section of the associated canonical bundle of $X_{k,\tau}$. 

For more general $\tau=(\tau_1,\tau_2,\tau_3,\tau_4)$, observe that the extension of $\tau_i$ 
 to a symplectomorphism $\tau_i:D^2_\epsilon\times T\times T\to D^2_\epsilon\times T\times T$ by the formula
\begin{equation}\label{taui}
\tau_i(x,y,e^{i z}, e^{iw},  \sigma_1,\sigma_2)=(x,y,e^{i (p_iz+q_iw)}, e^{i(r_iz+s_iw)},  \sigma_1,\sigma_2)
\end{equation}
induces a linear change of coordinate 1-forms $$ \tau_i^*(dz)=p_idz+r_idw,~ \tau_i^*(dw)= q_idz+s_idw$$ (and $\tau_i^*(dx)=dx,\tau_i^*(dy)=dy,\tau_i^*(d\sigma_1)=d\sigma_1,\tau_i^*(d\sigma_2)=d\sigma_1$).  The argument extends by replacing all occurences of  $dz$ and $ dw$  by $\tau_i^*(dz) $ and  $\tau_i^*(dw) $  in the definitions of $J_k, \omega_k$, and $s_k$. We leave the details to the reader. 

Thus the tangent bundle of $X_{k,\tau}$ admits an almost complex structure compatible with its symplectic form and a nowhere zero section of the associated canonical bundle of $(3,0)$ forms.  Thus $c_1(X_{k,\tau})=0$, as asserted.
\end{proof}

It is not necessarily true that the result of coisotropic Luttinger surgery along a 4-torus in a symplectic 6-manifold $X$ with $c_1(X)=0$ yields a manifold with vanishing first Chern class in general, see e.g. ~\cite{ADK}.  The important point in the preceding proof is that the non-vanishing section $s_X$ of the canonical bundle of $T^6$ over each 
$D^2_\epsilon\times T\times T$ coincides with the ``coordinate'' section $s_0=(dx+idz)(dw+idy)(d\sigma_1+id\sigma_2)$ via the embedding $e_i$. In the general case one may need to interpolate between the restriction of a given section of the canonical bundle of $X$ to the coordinate section over the neighborhood of the 4-torus.  This interpolation leads in general to the addition of ``rim'' 4-cycles, supported near the boundary of $D^2_\epsilon\times T\times T$,  to the divisor of the canonical  class of the surgered symplectic manifold.

\bigskip

 The following lemma will be used in the proof of Theorem \ref{product}.
 
 \begin{lem}\label{homology}  Let $N$ denote the union of the four tubular neighborhoods of $e_i(T^4)$ in $T^6$, where $e_i:T^4\to T^6$ are the embeddings of Equation (\ref{embeds}).  Then the inclusion 
 $$H_1(T^6 \setminus N)\to H_1(T^6)\cong\ZZ^6$$ is an isomorphism.  Moreover,   $H_2(T^6 \setminus N)\cong\ZZ^{17}$.
  \end{lem}
  
  \begin{proof} Thicken  the embeddings $e_i:T^4\to T^6$ to (closed) tubular neighborhoods $e_i:D^2_\epsilon\times T^4\to T^6$ and  denote the  union of these  tubular neighborhoods by $N$. 
The  excision and Kunneth theorems  give isomorphisms
$$\oplus_{i=1}^4 H_{n-2}(e_i(T^4))\cong H_{n-2} (\sqcup_i e_i(T^4))\otimes H_2(D^2,S^1)  \cong H_n(N,\partial N) \cong H_n(T^6, T^6 \setminus N).$$ 
This correspondence assigns to an $(n-2)$-cycle $\gamma\subset e_i(T^4) $ the  product $ \gamma\times (D^2,S^1)$.   The connecting homomorphism $H_n(T^6,T^6 \setminus N)\to H_{n-1}(T^6 \setminus N)$ takes $\gamma\times(D^2,S^1)$ to $\gamma\times S^1$.

 Consider the exact  sequence of the pair
 \begin{equation}\label{mv}
 \cdots\to H_{n+1}(T^6, T^6 \setminus N)\to H_n(T^6 \setminus N)\to H_n(T^6)\to H_n( T^6, T^6 \setminus N)\to\cdots
\end{equation}

Since $H_1(T^6,T^6 \setminus N)=0$,  $H_1(T^6 \setminus N)\to H_1(T^6)$ is surjective. The connecting homomorphism 
 $\ZZ^4\cong H_2(T^6, T^6 \setminus N)\to H_1(T^6 \setminus N)$ has image generated by the four meridians 
 $\mu_i= e_i(p)\times S^1$.    
 
 But  $\mu_i=0\in H_1(T^6 \setminus N)$ since they bound the punctured dual tori.   More explicitly,   $T_1=S^1\times\{-1\}\times \{-1\} \times S^1\times\{-1\}\times \{-1\}$ is a 2-dimensional torus in $T^6$ which meets $e_1(T^4)$ in the meridian disk $e_1(D^2_\epsilon\times\{-1\}\times\{-1\})$ and is disjoint from $e_i(T^4)$ for $i=2,3,4$.  Thus $T_1-e_1(D^2_\epsilon\times\{-1\}\times\{-1\})$ is a 2 chain in $T^6 \setminus e_1(T^4)$ with boundary $\mu_1$.  Similar arguments show that all the $\mu_i$ are zero.   Hence $H_1(T^6 \setminus N)\to H_1(T^6)\cong\ZZ^6$ is an isomorphism.

We claim the homomorphism 
$$H_3(T^6)\to H_3(T^6,T^6 \setminus N)\cong  \oplus_{i=1}^4H_1( e_i(T^4))   \cong    \ZZ^{16}$$ 
has rank  10.  

First note that it has rank at most 10, since $H_3(T^6)\cong \ZZ^{20}$, and the ${{6-1}\choose{3}}=10$
coordinate 3-tori with first coordinate fixed lift to $H_3(T\setminus N)$: just choose their first coordinate distinct from $\pm 1,\pm i$.

Denote by $W_i$, $i=1,2,\cdots ,10$  the following representatives of the remaining ten coordinate 3-tori in $T^6$:
\begin{equation*}
\begin{split}
W_1=\{(e^{ia}, e^{ib},-1,  e^{ic}, -1, -1)~|~ a,b,c\in \RR\}\\
W_2=\{(e^{ia}, e^{ib}, e^{ic}, -1, -1, -1)~|~ a,b,c\in \RR\}\\
W_3=\{(e^{ia}, e^{ib},-1, -1,  -1, e^{ic})~|~ a,b,c\in \RR\}\\
W_4=\{(e^{ia}, e^{ib},-1, -1,  e^{ic}, -1)~|~ a,b,c\in \RR\}\\
W_5=\{(e^{ia}, -1, e^{ib}, -1,  -1, e^{ic})~|~ a,b,c\in \RR\}\\
W_6=\{(e^{ia}, -1, e^{ib},-1,   e^{ic}, -1)~|~ a,b,c\in \RR\}\\
W_7=\{(e^{ia}, -1, e^{ib},   e^{ic} ,-1, -1)~|~ a,b,c\in \RR\}\\
W_8=\{(e^{ia}, -1, -1, e^{ib}, e^{ic}-1)~|~ a,b,c\in \RR\}\\
W_9=\{(e^{ia}, -1, -1, e^{ib},-1,   e^{ic})~|~ a,b,c\in \RR\}\\
W_{10}=\{(e^{ia}, -1, -1,-1,e^{ib}, e^{ic})~|~ a,b,c\in \RR\}\\
\end{split}
\end{equation*}
These generate a rank 10 free abelian subgroup of $H_3(T^6)$, and intersect each of the $e_i(T^4)$ transversely.  Thus the image of each of these ten 3-cycles in $H_3(T^6,T^6 \setminus N)\cong\oplus_{i=1}^4H_1(e_i(T^4))$ is determined by taking its (transverse) intersection with $e_i(T^4)$.

For example,  $W_1$ misses $e_j(T^4)$ for $j\ne 1$ and intersects $e_1(T^4)$ transversely in the  homologically essential circle $ \{1\}\times S^1\times  (-1,1,-1,-1) =e_1(S^1\times (-1,-1,-1) )$. Similarly, for $1\leq i,j \leq 4$, $W_j$ misses $e_i(T^4)$ when $i\ne j$ and intersects $e_i(T^4)$ in a homologically essential circle.

For $5\leq j\leq10$,   $W_j$ intersects exactly two of the $e_i(T^4)$. For example, $W_5 $ is disjoint from $e_1(T^4)$ and $e_4(T^4)$ and intersects $e_2(T^4)$ in the circle
$ (i,-1,1,-1,-1) \times S^1=e_2((-1,-1,-1)\times S^1)$, and intersects $e_3(T^4)$ in the circle
$(-1,-1)\times S^1\times(-1,-1,1)=e_3((-1,-1)\times S^1\times -1)$.

We leave the reader the straightforward check that the 10  cycles in $\oplus_{i=1}^4 H_1(e_i(T^4))\cong \ZZ^{16}$ are linearly independent and span a summand.  Thus the rank of $H_3(T^6)\to H^3(T^6, T^6\setminus N)$ is 10 and its cokernel is $\ZZ^6$.

 From the exact  sequence (\ref{mv}) with $n=2$ we obtain
 $$0\to \ZZ^{6}\to H_2(T^6 \setminus N)\to H_2(T^6)\to \ZZ^4\to 0.$$
Since $H_2(T^6)\cong \ZZ^{15}$ and $H_1(T^6\setminus N)$ is free abelian, we conclude that   $H_2(T^6 \setminus N)\cong\ZZ^{17}$. \end{proof}

\bigskip

Lemma \ref{homology} says that $H_1(T^6 \setminus N)\to H_1(T^6)$ is an isomorphism and hence the six coordinate circles freely generate  $H_1(T^6 \setminus N)\cong\ZZ^6$.  Label these generators $x_1,\cdots, x_6$.  As explicit curves in $T^6 \setminus N$, one can take $x_1=S^1\times (p,p,p,p,p), $  $x_2= p \times S^1\times (p,p,p,p)$, etc., where $p$ is a primitive eight root of unity.

\begin{prop} \label{rels} Given $(k,\tau)=(k_1, \tau_1; k_2,\tau_2;k_3,\tau_3;k_4,\tau_4)\in \ZZ^4\times (SL(2,\ZZ))^4$, with   $$\tau_i=\begin{pmatrix} p_i&q_i\\ r_i & s_i\end{pmatrix}, $$  then 
$H_1(X_{k,\tau})$ is the quotient of $\ZZ^6=H_1(T^6 \setminus N)=\langle x_1,x_2,x_3,x_4,x_5,x_6\rangle$ by the subgroup generated by 
$$k_1(q_1x_2+s_1x_3), k_2(q_2x_2+s_2x_4), k_3(q_3x_2+s_3x_5),k_4(q_4x_2+s_4x_6).$$

Moreover, the Betti numbers of $X_{k,\tau}$ satisfy
 $$b_2(X_{k,\tau})\leq 15+b_1(X_{k,\tau})~\text{and}~
   b_3(X_{k,\tau})\leq 32.$$

\end{prop}

\begin{proof}  
The manifold $D^2\times T^4$ has its usual handle decomposition with one 0-handle, four 1-handles, six 2-handles, four 3-handles, and one 4-handle. Turning it upside down one obtains the dual handle decomposition, showing that $D^2\times T^4$ is obtained from $\partial (D^2)\times T^4\times [0,1]$ by adding one 2-handle along $\mu=\partial (D^2)\times \{p\}\times\{1\}$, then adding four 3-handles, six 4-handles, four 5-handles, and one 6-handle.

Hence $H_1(X_{k,\tau})$ is the quotient of $H_1(T^6 \setminus N)$ by the subgroup generated by the four 
circles   along which the 2-handles are reattached in passing from $T^6\setminus N$ to $X_{k,\tau}$.

From the formulas  (\ref{psi}), (\ref{taui}), and (\ref{embeds}) one sees that
the   2-handle corresponding to $e_i$ is attached along $e_i\circ\tau_i \circ \psi_{k_i}(\partial D^2_\epsilon)$. In terms of the generators $x_1\cdots, x_6$ of $H_1(T^6 \setminus N)$, a simple calculation shows that 
 \begin{equation}\label{relations} \begin{split}
[e_1(\tau_1(\psi_{k_1}(\partial D^2_\epsilon)))]=k_1(q_1x_2+s_1x_3)\\
[e_2(\tau_2(\psi_{k_2}(\partial D^2_\epsilon)))]=k_2(q_2x_2+s_2x_4)\\
[e_3(\tau_3(\psi_{k_3}(\partial D^2_\epsilon)))]=k_3(q_3x_2+s_3x_5)\\
[e_4(\tau_4(\psi_{k_4}(\partial D^2_\epsilon)))]=k_4(q_4x_2+s_4x_6)
 \end{split}
\end{equation}

 Thus $H_1(X_{k,\tau})$ is isomorphic to the quotient of $\ZZ^6$ by the four 1-cycles on the right side of Equation (\ref{relations}).
 
 Lemma \ref{homology} shows that $H_2(T^6\setminus N)\cong \ZZ^{17}$. Attaching a 2-handle to a manifold increases the second Betti number if and only if the attaching circle has finite order in  first homology. Hence  if $H_1(X_{k,\tau})\cong\ZZ^{6-d}\oplus F$ for a finite abelian group $F$, 
 the second Betti number of $T^6\setminus N$ with the  four 2-handles attached is $17+(4-d)=21-d$.  Attaching the sixteen 3-handles decreases the second Betti number further, and the 4-handles, 5-handles, and 6-handles do not change the second Betti number.  Hence
 $b_2(X_{k,\tau})\leq 21-d= 15+ b_1(X_{k,\tau})$.  The Euler characteristic of $X_{k,\tau}$ equals zero, and so
 $$0=2-2b_1(X_{k,\tau})+2b_2(X_{k,\tau})-b_3(X_{k,\tau})\leq 32-b_3(X_{k,\tau}).$$
Therefore $b_3(X_{k,\tau})\leq 32$. 
\end{proof}

 By choosing the $\tau_i$ and $k_i$ appropriately, one can ensure that $H_1(X_{k,\tau})$ is    isomorphic to  
 \begin{equation}\label{H1}
\ZZ^2\oplus \ZZ/d_1\oplus \ZZ/d_2\oplus \ZZ/d_3\oplus \ZZ/d_4 
\end{equation}
 for any  4-tuple of non-negative integers $d_i$ (e.g.~ take $\tau_i=$Id and $k_i=d_i $).
  In particular, when an odd number of the $d_i$ are non-zero,  then  the first Betti number is is odd and hence $X_{k, \tau}$ cannot be K\"ahler.
  
\bigskip
 
To ensure that our construction produces new manifolds, we have the following. 
 \begin{thm} \label{product}  If $H_1(X_{k,\tau})$ has rank  2 or 3, then $X_{k,\tau}$ is not symplectomorphic to the product of any symplectic 4-manifold with a surface.
 \end{thm}
 
 \begin{proof}
Choose an $X= X_{k,\tau}$ such that the first Betti number of $X$ satisfies $b_1(X)= 2+r$ for $r=0$ or $1$.    Lemma \ref{homology} shows that $b_2(X)\leq 18$.

 Suppose that $X$ were  symplectomorphic to $M\times F$, for some symplectic 4-manifold $M$ and closed oriented surface $F$.  Then $0=c_1(X)=\pi_1^*(c_1(M))+\pi_2^*(c_1(F))$, where $\pi_i$ denote the projections of $M\times F$ to its two factors.  The Kunneth theorem shows that $\pi_1^*+ \pi_2^*:H^2(M)\oplus H^2(F)\to H^2(M\times F)$ is injective, and hence $c_1(M)=0$ and $c_1(F)=0$.  Thus $F$ is a torus, $F=T^2$, and $M$ admits a Spin structure.

  Rohlin's theorem then shows that
 the signature $\sigma(M)=b^+(M)-b^-(M)$ is a multiple of 16, and so 
 $$b^-(M)=b^+(M)+16n$$ for some integer $n$.  The Kunneth theorem implies that $b_1(M)=r$, and so the Euler characteristic 
 is given by $e(M)=2-2r+b_2(M)= 2-2r+b^+(M)+ b^-(M)$.  Since  $0=c_1(M)^2=2e(M)+3\sigma(M)=4-4r+ 5b^+(M)-b^-(M),$   $$b^-(M)=5b^+(M)+4-4r$$ and we conclude that 
 $$b^+(M)=4n+r-1.$$ 
 
 The symplectic form satisfies $\omega_M^2>0$ and so $b^+(M)\ge 1$.  Since $r=0$ or $1$,   this implies that $n\ge 1$.  Hence
 $$b_2(M)=b^+(M)+b^-(M)= 6 b^+(M)+ 4 -4r=24n +2r-2\ge 22.$$
 Then 
 $$b_2(X)= b_2(M)+ b_1(M) b_1(T^2)+ b_2(T^2)   \ge 23.$$
   This contradicts the bound $b_2(X)\leq 18$ obtained above.
 
 \end{proof}
 
  Arguments like those given in the proof of Theorem  \ref{product} can be used to show that   $X_{k,\tau}$ is not homotopy equivalent to the product of a symplectic  4-manifold with a surface of genus 2 or more for any $(k, \tau)$. This leaves open the possibility that some $X_{k,\tau}$ is homotopy equivalent or even diffeomorphic  to $M\times S^2$.

  \bigskip

 We do not know if every $X_{k, \tau}$ with even first Betti number  admits a \Kahler structure, but conjecture that most do not.  The reason for this conjecture is that most $X_{k, \tau}$ are not likely to satisfy the Hard Lefschetz Theorem (cf.  \cite{B, BL}).

\section{Concluding remarks}

Coisotropic Luttinger surgery can be useful in other contexts. For example, an easy extension of Theorem \ref{result} can be obtained by considering surgeries on $2n$-dimensional tori.  An interesting setting occurs when a symplectic manifold fibers in several different ways. In our   examples  we applied this to  $T^6=T^2\times T^2\times T^2$ and its three coordinate fibrations to the 4-torus. One could also  start with a product of closed surfaces $X=\Sigma_{g_1}\times \Sigma_{g_2}\times \cdots\times \Sigma_{g_n}$, which contains many coisotropic submanifolds of the form $T^2\times Z$  obtained as   preimages of  Lagrangian tori with respect to various projections of  $X $  to $\Sigma_{g_i}\times \Sigma_{g_j}$.

One can produce symplectic $2n$-manifolds with a wide range of possible homology groups and canonical classes by this method.  Deriving more concrete homotopy or diffeomorphism information is more difficult, as establishing control of the fundamental group is always a challenge in codimension two surgery constructions, c.f.~\cite{BK2}.

Another promising direction is to apply the method to  Lefschetz  fibrations. For example, the $K3$ surface, as a desingularization of $T^4/\ZZ/2$,  has  admits different elliptic fibrations. Thus $K3\times T^2$ contains   submanifolds on which one can perform coisotropic Luttinger surgery. This should lead to examples with smaller first homology and perhaps even to simply connected examples.




\end{document}